\documentclass{article}

\usepackage{amsfonts}
\usepackage{amsmath}
\usepackage{amssymb}
\usepackage{graphicx}

\newtheorem{theorem}{Theorem}[section]

\newtheorem{assumption}[theorem]{Assumption}

\newtheorem{definition}[theorem]{Definition}

\newtheorem{lemma}[theorem]{Lemma}

\newtheorem{proposition}[theorem]{Proposition}
\newtheorem{remark}[theorem]{Remark}

\newenvironment{proof}[1][Proof]{\noindent\textbf{#1.} }{\ \rule{0.5em}{0.5em}}

\makeatletter
   
   \@addtoreset{equation}{section}
\makeatother

\begin{document}

\title{Stochastic maximum principle for infinite dimensional control systems}
\author{Kai~Du\footnote{Email addresses: \texttt{kai.du@math.ethz.ch} (K.~Du),
\texttt{mqx@hutc.zj.cn} (Q.~Meng)} \and Qingxin~Meng\footnotemark[1]}
\date{}
\maketitle

\vspace{-10pt}
\begin{abstract}
The general maximum principle is proved for an infinite dimensional
controlled stochastic evolution system. The control is allowed to take values
in a nonconvex set and enter into both drift and diffusion terms. The
operator-valued backward stochastic differential equation, which
characterizes the second-order adjoint process, is understood via the concept
of ``generalized solution'' proposed by Guatteri and Tessitore [SICON 44
(2006)].

\medskip

\smallskip


\end{abstract}

\section{Introduction}

\subsection{Problem formulation and basic assumptions}

In this paper we shall always indicate by $H$ a real separable Hilbert space
and by $\langle\cdot,\cdot\rangle_{H}$ and $\Vert\cdot\Vert_{H}$ its inner
scalar product and norm, respectively. Denote by $B(H)$ the Banach space of
all bounded linear operators from $H$ to itself endowed with the norm $\Vert
T\Vert_{B(H)}:=\sup\{\Vert Tx\Vert_{H}:\Vert x\Vert_{H}=1\}$.

Let $(\Omega,\mathcal{F},\mathbb{F},\mathbb{P})$ be a probability space with
the filtration $\mathbb{F}=(\mathcal{F}_{t})_{t\geq0}$ generated by countable
independent standard Wiener processes $\{W^{i};i\in\mathbb{N}\}$ and augmented
with all $\mathbb{P}$-null sets of $\mathcal{F}$. For simplicity, we write
formally $f\cdot\mathrm{d}W=\sum_{i\in\mathbb{N}}f_{i}\,\mathrm{d}W^{i}$ with
a sequence $f=(f_{i};i\in\mathbb{N})$. We denote by $\mathbb{E}^{\mathcal{F}%
_{t}}$ the conditional expectation with respect to $\mathcal{F}_{t}$, and by
$\mathcal{P}$ the predictable $\sigma$-field on $\Omega\times\lbrack0,1]$.

In this paper, we study an infinite-dimensional optimal control problem
governed by the following abstract semilinear stochastic evolution equation (SEE)%
\begin{align}
\mathrm{d}x(t)  &  =\left[  Ax(t)+f(t,x(t),u(t))\right]  \mathrm{d}%
t+g(t,x(t),u(t))\cdot\mathrm{d}W_{t},\nonumber\label{eq:stateeq}\\
x(0)  &  =x_{0},
\end{align}
where $x(\cdot)$ is the state process and $u(\cdot)$ is the control. The
\emph{control set} $U$ is a nonempty Borel-measurable subset of a metric space
whose metric is denoted by $\mathrm{dist}(\cdot,\cdot)$. Fix an element
(denoted by $0$) in $U$, and then define $|u|_{U}=\mathrm{dist}(u,0)$. An
\emph{admissible control} $u(\cdot)$ is a $U$-valued predictable process such
that%
\[
\sup\big\{\mathbb{E}\left\vert u(t)\right\vert _{U}^{4}:t\in\lbrack
0,1]\big\}<\infty.
\]
Our optimal control problem is to find an admissible control $u(\cdot)$
minimizing the cost functional
\[
J(u(\cdot))=\mathbb{E}\int_{0}^{1}l(t,x(t),u(t))\,\mathrm{d}t+\mathbb{E}%
h(x(1)).
\]
In the above statement, $A$ is the infinitesimal generator of a $C_{0}%
$-semigroup, and
\begin{align*}
f  &  :\Omega\times\lbrack0,1]\times H\times U\rightarrow H,\ \ g:\Omega
\times\lbrack0,1]\times H\times U\rightarrow l^{2}(H),\\
l  &  :\Omega\times\lbrack0,1]\times H\times U\rightarrow\mathbb{R}%
,\ \ h:\Omega\times H\rightarrow\mathbb{R},
\end{align*}
where the Hilbert space
\[
l^{2}(H):=\Big\{z=(z_{i};i\in\mathbb{N}):\Vert z\Vert_{l^{2}(H)}^{2}%
:=\sum\nolimits_{i\in\mathbb{N}}\Vert z_{i}\Vert_{H}^{2}<\infty\Big\}.
\]

Throughout this paper we make the following assumptions.

\begin{assumption}
\label{ass:A} The operator $A:D(A)\subset H\rightarrow H$ is the infinitesimal
generator of a $C_{0}$-semigroup $\{\mathrm{e}^{tA}\in B(H);t\geq0\}$. Set%
\[
M_{A}:=\sup\{\Vert\mathrm{e}^{tA}\Vert_{B(H)}:t\in\lbrack0,1]\}.
\]

\end{assumption}

\begin{assumption}
\label{ass:fglh} The functions $f$, $g$ and $l$ are all $\mathcal{P}%
\times\mathcal{B}(H)\times\mathcal{B}(U)$-measurable, and $h$ is
$\mathcal{F}_{1}\times\mathcal{B}(H)$-measurable; for each $(t,u,\omega
)\in\lbrack0,1]\times U\times\Omega$, $f,g,l$ and $h$ are globally twice
Fr\`{e}chet differentiable with respect to $x$; $f_{x},g_{x},f_{xx}%
,g_{xx},l_{xx}$ and $h_{xx}$ are bounded by a constant $M_{0}$; $f,g,l_{x}$
and $h_{x}$ are bounded by $M_{0}(1+\Vert x\Vert_{H}+|u|_{U})$; $l$ and $h$ is
bounded by $M_{0}(1+\Vert x\Vert_{H}^{2}+|u|_{U}^{2})$.\noindent
\end{assumption}

In view of Assumption \ref{ass:A}, the precise meaning of the equation
\eqref{eq:stateeq} is
\[
x({t})=\mathrm{e}^{tA}x_{0}+\int_{0}^{t}\mathrm{e}^{(t-s)A}\big[f(s,x({s}%
),u(s))\,\mathrm{d}s+g(s,x({s}),u(s))\cdot\mathrm{d}W_{s}\big].
\]
A process $x(\cdot)$ satisfying the above equality is usually called a
\emph{mild solution} to equation \eqref{eq:stateeq}, cf.
\cite{da1992stochastic}.

\subsection{Developments of stochastic maximum principle and contributions of the paper}

The aim of this work is to find a stochastic maximum principle (SMP for
short) for the optimal control. As we know, SMP is one of the basic tools to
study optimal stochastic control problems. Since \cite{kushner1972necessary},
there have been a number of results on such a subject. For finite dimensional
systems\footnote{For more related studies, we refer to
\cite{yong1999stochastic} and the references therein.}, the problem in the
general case was solved by Peng \cite{peng1990general}. Hereafter, by the
word ``general'' we mean the allowance of the control into the diffusion term
and the nonconvexity of control domains. In contrast, most existing results
on infinite dimensional systems are limited to the case in which the control
domain is convex or the diffusion does not depend on the control (cf.
\cite{bensoussan1983stochastic, hu1990maximum,tang1993maximum}). Recently,
several works \cite{lvzhangmaximum,fuhrman2012stochastic,du2012general} were
devoted to the general SMP in infinite dimensions. L\"u-Zhang
\cite{lvzhangmaximum} first addressed such a problem and formulated a general
SMP in which they assumed the existence of the second-order adjoint process.
Fuhrman et al. \cite{fuhrman2012stochastic} focused on a concrete equation
(which was a stochastic parabolic PDE with deterministic coefficients) and
gave a complete formulation of SMP, while their approach depended on the
special structure of the equation. In our previous work \cite{du2012general},
a general SMP was obtained  for abstract stochastic parabolic equations
driven by finite Wiener processes. As far as we know, the general SMP for
stochastic evolution equations as form \eqref{eq:stateeq} is not completely
solved.

Our basic approach to derive the general SMP follows Peng's idea of
second-order expansion in calculating the variation of the cost functional
caused by the spike variation. The key point is how to understand and solve a
$B(H)$-valued backward stochastic differential equation (BSDE) which
characterizes the second-order adjoint process in our SMP. To do this, we
exploit a concept of solution to this equation called ``generalized
solution'' which was first proposed by Guatteri-Tessitore
\cite{guatteri2006backward} in the study of infinite dimensional LQ problems,
and prove the existence-uniqueness result in our framework. However, the
generalized solution only characterizes the first unknown
component\footnote{Normally, the solution of a BSDE consists of a pair of
adapted processes of which the second one is the diffusion term. For more
aspects on BSDE, we refer to \cite{el2008backward}.} but says nothing about
the second one. As a consequence, it seems difficult to derive our SMP via
the traditional approach (i.e. applying Ito's formula). To avoid this
difficulty, we first derive a basic property of the generalized solution.
Then our goal is achieved thanks to the Lebesgue differentiation theorem.

The rest of this paper is organized as follows. Section 2 is devoted to some
preliminary results on SEE and backward SEE (BSEE) in the Hilbert space. In
Section 3, we study the well-posedness and basic properties of an
operator-valued BSDE. With the previous preparations, we shall state and
prove our main theorem, the stochastic maximum principle, in the final
section.

\medskip We finish the introduction with some notations. Let $H$ be a
separable Hilbert space and $\mathcal{B}(H)$ be its Borel $\sigma$-field. The
following classes of processes will be used in this article. Here
$p,q\in\lbrack1,\infty]$.

$\bullet\ \ L_{\mathcal{P}}^{p}(\Omega\times\lbrack0,1];H)$ denotes the space
of equivalence classes of processes $x(\cdot)$, admitting a predictable
version such that $\mathbb{E}\int_{0}^{1}\Vert x(t)\Vert_{H}^{p}
\mathrm{d}t<\infty$.

$\bullet\ \ \mathcal{C}_{\mathcal{P}}([0,1];L^{p}(\Omega;H))$ denotes the
space of $H$-valued processes $x(\cdot)$ such that $x(\cdot):[0,1]\rightarrow
L^{p}(\Omega;H)$ is continuous and has a predictable modification.
%

Moreover, since the $\sigma$-field generated by the operator norm in $B(H)$ is
too large, we shall define the following spaces with respect to $B(H)$-valued
processes and random variables.

$\bullet\ \ L_{\mathcal{P},\mathrm{S}}^{p}(\Omega\times\lbrack0,1];B(H))$
denotes the space of equivalence classes of $B(H)$-valued processes $T(\cdot)$
such that $T(\cdot)x\in L_{\mathcal{P}}^{p}(%
\Omega
\times\lbrack0,1];H)$ for each $x\in H$. Here the subscript \textquotedblleft%
$\mathrm{S}$\textquotedblright\ stands for \textquotedblleft strongly
measurable\textquotedblright.

$\bullet\ \ L_{\mathcal{F}_{t},\mathrm{S}}^{p}(\Omega;B(H))$ denotes the space
of equivalence classes of $B(H)$-valued random variable $T$ such that $Tx\in
L_{\mathcal{F}_{t}}^{p}(\Omega;H)$ for each $x\in H$.


\section{Preliminary results on SEEs and BSEEs}

Let $b:~\Omega\times\lbrack0,T]\times H\rightarrow H$ and $\sigma
:~\Omega\times\lbrack0,T]\times H\rightarrow l^{2}(H)$ be two
$\mathcal{P\times}\mathcal{B}(H)$-measurable mappings and $F:\Omega
\times\lbrack0,T]\times H\times l^{2}(H)\rightarrow H$ be a $\mathcal{P\times
B}(H)\times\mathcal{B}(l^{2}(H))$-measurable mapping such that
\begin{align*}
&  \Vert b(t,x)-b(t,\bar{x})\Vert_{H}+\Vert\sigma(t,x)-\sigma(t,\bar{x}%
)\Vert_{l^{2}(H)}\leq M_{1}\Vert x-\bar{x}\Vert_{H},\\
&  \Vert F(t,x,y)-F(t,\bar{x},\bar{y})\Vert_{H}\leq M_{1}\left(  \Vert
x-\bar{x}\Vert_{H}+\Vert y-\bar{y}\Vert_{l^{2}(H)}\right)  \ \ \text{(a.s.)}%
\end{align*}
for some constant $M_{1}>0$ and any $t\in\lbrack0,1]$, $x,\bar{x}\in H$ and
$y,\bar{y}\in l^{2}(H)$.

For given operator $A$ satisfying Assumption \ref{ass:A}, consider the
following SEE
\begin{equation}
x({t})=\mathrm{e}^{tA}x_{0}+\int_{0}^{t}\mathrm{e}^{(t-s)A}\big[b(s,x({s}%
))\,\mathrm{d}s+\sigma(s,x({s}))\cdot\mathrm{d}W_{s}\big] \label{eq:SEE}%
\end{equation}
and BSEE%
\begin{equation}
p({t})=\mathrm{e}^{(1-t)A^{\ast}}\xi+\int_{t}^{1}\mathrm{e}^{(s-t)A^{\ast}%
}\big[F(s,p({s}),q(s))\,\mathrm{d}s-q(s)\cdot\mathrm{d}W_{s}\big].
\label{eq:bsee}%
\end{equation}

Now we present some preliminary results on SEE \eqref{eq:SEE} and BSEE
\eqref
{eq:bsee} which will be often used in this paper.

\begin{lemma}
\label{lem:see} Under the above setting, we have the following assertions:

(1) If $p\in\lbrack2,\infty)$, $b(\cdot,0)\in L_{\mathcal{P}}^{p}(\Omega
\times\lbrack0,1];H)$ and $\sigma(\cdot,0)\in L_{\mathcal{P}}^{p}(\Omega
\times\lbrack0,1];l^{2}(H))$, then SEE \eqref{eq:SEE} has a unique solution
$x(\cdot)$ in the space $\mathcal{C}_{\mathcal{P}}([0,1];L^{p}(\Omega;H))$ for
any given $x_{0}\in H$, with the $L^{p}$-estimate
\begin{align*}
\mathbb{E}\sup_{t\in\lbrack0,1]}\left\Vert x(t)\right\Vert _{H}^{p}~\leq~  &
K(M_{A},M_{1},p)\,\mathbb{E}\bigg[\left\Vert x_{0}\right\Vert _{H}%
^{p}+\bigg(\int_{0}^{1}\left\Vert b(t,0)\right\Vert _{H}\mathrm{d}%
t\bigg)^{p}\\
&  +\bigg(\int_{0}^{1}\left\Vert \sigma(t,0)\right\Vert _{l^{2}(H)}%
^{2}\mathrm{d}t\bigg)^{p/2}\bigg].
\end{align*}

(2) If $F(\cdot,0,0)\in L_{\mathcal{P}}^{2}(\Omega\times\lbrack0,1];H)$, then
BSEE \eqref{eq:bsee} has a unique solution $(p(\cdot),q(\cdot))$ in the space%
\[
\mathcal{C}_{\mathcal{P}}([0,1];L^{2}(\Omega;H))\times L_{\mathcal{P}}%
^{2}(\Omega\times\lbrack0,1];l^{2}(H))
\]
for any given $\xi\in L_{\mathcal{F}_{1}}^{2}(\Omega;H)$, with the estimate
\begin{align*}
&  \mathbb{E}\sup_{t\in\lbrack0,1]}\left\Vert p(t)\right\Vert _{H}%
^{2}+\mathbb{E}\int_{0}^{1}\left\Vert q(t)\right\Vert _{l^{2}(H)}
^{2}\mathrm{d}t\\
&  \leq~K(M_{A},M_{1})\,\mathbb{E}\left[  \left\Vert \xi\right\Vert _{H}%
^{2}+\int_{0}^{1}\left\Vert F(t,0,0\right)  \Vert_{H}^{2}\,\mathrm{d}t\right]
.
\end{align*}

Hereafter $K(\cdot)$ is a positive constant depending only on the values in
the brackets.
\end{lemma}

The above results can be found in, for example,
\cite{da1992stochastic,hu1991adapted}. The following lemma, concerning the
duality between SEE \eqref{eq:SEE} and BSEE \eqref{eq:bsee}, can be easily
derived by the Yoshida approximation (cf. \cite{tessitore1996existence}).

\begin{lemma}
\label{lem:duality} Under the conditions in Lemma \ref{lem:see}, we have%
\begin{align*}
&  \mathbb{E}\left\langle x(t_{1}),p(t_{1})\right\rangle _{H}+\mathbb{E}%
\int_{t_{1}}^{t_{2}}\left[  \left\langle b(s,x(s)),p(s)\right\rangle
_{H}+\left\langle \sigma(s,x(s)),q(s)\right\rangle _{l^{2}(H)}\right]
\mathrm{d}s\\
&  =\mathbb{E}\left\langle x(t_{2}),p(t_{2})\right\rangle _{H}+\mathbb{E}%
\int_{t_{1}}^{t_{2}}\left\langle x(s),F(s,p(s),q(s))\right\rangle
_{H}\mathrm{d}s
\end{align*}
for any $0\leq t_{1}\leq t_{2}\leq1$.
\end{lemma}

\section{Operator-valued BSDEs: well-posedness and properties}

In this section, we study the following operator-valued BSDE (OBSDE)
\begin{align}
\mathrm{d}P(t)\,=\,  &  -\big\{A^{\ast}P(t)+P(t)A+A_{\sharp}^{\ast
}(t)P(t)+P(t)A_{\sharp}(t)\nonumber\label{eq:obsde}\\
&  +\mathrm{Tr[}C^{\ast}PC+QC+C^{\ast}Q](t)+G(t)\big\}\mathrm{d}%
t+Q(t)\cdot\mathrm{d}W_{t},\nonumber\\
P(1)\,=\,  &  P_{1}%
\end{align}
with the unknown processes $P(\cdot)$ and $Q(\cdot)$, where $A$ satisfies
Assumption \ref{ass:A}, $A_{\sharp}(\cdot)$ and $C(\cdot)$ are given
coefficients, and
\[
\mathrm{Tr}[C^{\ast}PC+QC+C^{\ast}Q](t)=\sum\nolimits_{i\in\mathbb{N}}
[C_{i}^{\ast}PC_{i}+Q_{i}C_{i}+C_{i}^{\ast}Q_{i}](t).
\]
We call the pair $(G,P_{1})$ the \emph{input} of OBSDE \eqref{eq:obsde}. Such
a equation will be used to characterize the second order adjoint process of
the controlled system in the next section. We make the following assumption.

\begin{assumption}
\label{ass:obsde} $A_{\sharp}(\cdot)\in L_{\mathcal{P},\mathrm{S}}^{\infty
}(\Omega\times\lbrack0,1];B(H))$. $C(\cdot)=(C_{i}(\cdot)\,;\,i\in\mathbb{N})$
with $C_{i}(\cdot)\in L_{\mathcal{P},\mathrm{S}}^{\infty}(\Omega\times
\lbrack0,1];B(H))$. Assume
\[
M_{2}:=\mathop{\mathrm{ess}\sup}\nolimits_{t,\omega}\Big\{\Vert A_{\sharp
}(t,\omega)\Vert_{B(H)}^{2},\ \sum\nolimits_{i\in\mathbb{N}}\Vert
C_{i}(t,\omega)\Vert_{B(H)}^{2}\Big\}<\infty.
\]

\end{assumption}

\subsection{The well-posedness}

The solvability theory of $B(H)$-valued BSDEs is still far from complete. A
first remarkable work on such a subject was found in Guatteri-Tessitore
\cite{guatteri2006backward} where, inspired by the notion of ``strong
solution'' for PDEs (cf. \cite{bensoussan1993representation}), the authors
proposed for an OBSDE the concept of \emph{generalized solution} which only
involved the first unknown $P(\cdot)$, and obtained the corresponding
existence-uniqueness result. Their approach based on the solvability of
\eqref{eq:obsde} in the Hilbert space $B_{2}(H)$ of all Hilbert-Schmidt
operators from $H$ to itself, see Theorem 5.4 in \cite{guatteri2006backward}.

Following the spirit of Guatteri-Tessitore \cite{guatteri2006backward}, we
give the following definition.

\begin{definition}
[generalized solution]\label{defn:ger}$P(\cdot)$ is called a \emph{generalized
solution} to OBSDE \eqref{eq:obsde} in $L_{\mathcal{P},\mathrm{S}}^{2}%
(\Omega\times\lbrack0,1];B(H))$ if there exists a sequence $(P^{n},Q^{n}%
,G^{n})$ such that

\emph{1)} $P^{n}(1)\in L_{\mathcal{F}_{1}}^{2}(\Omega;B_{2}(H))$, $G^{n}\in
L_{\mathcal{P}}^{2}(\Omega\times\lbrack0,1];B_{2}(H))$, and there exists a
constant $\lambda\geq1$ such that
\[
\Vert P^{n}(1)\Vert_{B(H)}\leq\lambda\Vert P_{1}\Vert_{B(H)}\text{\ \ and\ \ }%
\Vert G^{n}(t)\Vert_{B(H)}\leq\lambda\Vert G(t)\Vert_{B(H)}\ \ \text{(a.s.)}%
\]
for any $n\in\mathbb{N}$ and $t\in\lbrack0,1]$.

\emph{2)} $(P^{n},Q^{n})$ is a \emph{mild solution}\footnote{See Definition
5.3 in \cite{guatteri2006backward}.} to OBSDE \eqref{eq:obsde} with the input
$(G^{n},P^{n}(1))$, that is, for all $t\in\lbrack0,1]$,
\begin{align}
\label{eq:mild}P^{n}(t)\,=\,  &  \mathrm{e}^{(1-t)A^{\ast}}P^{n}%
(1)\mathrm{e}^{(1-t)A}+\int_{t}^{1}\mathrm{e}^{(s-t)A^{\ast}}[A_{\sharp}%
^{\ast}P^{n}+P^{n}A_{\sharp}+G^{n}](s)\mathrm{e}^{(s-t)A}\,\mathrm{d}%
s\nonumber\\
&  +\int_{t}^{1}\mathrm{e}^{(s-t)A^{\ast}}\mathrm{Tr[}C^{\ast}P^{n}%
C+Q^{n}C+C^{\ast}Q^{n}](s)\mathrm{e}^{(s-t)A}\,\mathrm{d}s\nonumber\\
&  +\int_{t}^{1}\mathrm{e}^{(s-t)A^{\ast}}Q^{n}(s)\mathrm{e}^{(s-t)A}%
\cdot\mathrm{d}W_{s},\ \ \text{\ (a.s.).}%
\end{align}

\emph{3)} for any $x,y\in H$ and $t\in\lbrack0,1]$,%
\[
\langle x,G^{n}(t)y\rangle_{H}\rightarrow\langle x,G(t)y\rangle_{H}%
\ \ \text{and}\ \ \langle x,P^{n}(t)y\rangle_{H}\rightarrow\langle
x,P(t)y\rangle_{H}\ \ \text{(a.s.)}%
\]

\end{definition}

In the above definition only the process $P(\cdot)$ is characterized.
Nevertheless, this is sufficient and even natural for the optimal control
theory since $Q(\cdot)$ is not involved in the formulation of the SMP (see
Theorem \ref{thm:smp}). For more detailed account we refer to Remark 6.3 in
\cite{guatteri2006backward}. Now we give the following well-posedness result
on the generalized solution to OBSDE \eqref{eq:obsde}.

\begin{theorem}
\label{thm:obsde} Let Assumptions \ref{ass:A} and \ref{ass:obsde} be
satisfied. Suppose $P_{1}\in L_{\mathcal{F}_{1},\mathrm{S}}^{2}(\Omega;B(H))$
and $G\in L_{\mathcal{P},\mathrm{S}}^{2}(\Omega\times\lbrack0,1];B(H))$. Then

\emph{i)} there exists a unique {generalized solution} $P(\cdot)$ to OBSDE \eqref{eq:obsde};

\emph{ii)} for each $\tau\in\lbrack0,1]$ and any $\xi,\zeta\in L_{\mathcal{F}%
_{\tau}}^{4}(\Omega;H)$,
\begin{align}
\left\langle \xi,P(\tau)\zeta\right\rangle _{H}~=~  &  \mathbb{E}%
^{\mathcal{F}_{\tau}}\left\langle y^{\tau,\xi}(1),P_{1}y^{\tau,\zeta
}(1)\right\rangle _{H}\nonumber\label{eq:exprP}\\
&  +\mathbb{E}^{\mathcal{F}_{\tau}}\int_{\tau}^{1}\left\langle y^{\tau,\xi
}(t),G(t)y^{\tau,\zeta}(t)\right\rangle _{H}\mathrm{d}t\ \ \text{(a.s.)}%
\end{align}
with $y^{\tau,\xi}$ (similarly for $y^{\tau,\zeta}$) being the solution to
equation
\begin{align}
y^{\tau,\xi}(t)~=~  &  \mathrm{e}^{(t-\tau)A}\xi+\int_{\tau}^{t}%
\mathrm{e}^{(t-s)A}A_{\sharp}(s)y^{\tau,\xi}(s)\mathrm{d}%
s\nonumber\label{eq:y.eq}\\
&  +\int_{\tau}^{t}\mathrm{e}^{(t-s)A}C(s)y^{\tau,\xi}(s)\cdot\mathrm{d}%
W_{s},\ \ t\in\lbrack\tau,1];
\end{align}

\emph{iii)} for each $\tau\in\lbrack0,1]$, it holds almost surely that
\begin{equation}
\left\Vert P(\tau)\right\Vert _{B(H)}^{2}\leq K(M_{A},M_{2})\,\mathbb{E}%
^{\mathcal{F}_{\tau}}\bigg[ \left\Vert P_{1}\right\Vert _{B(H)}^{2}+\int
_{0}^{1}\left\Vert G(t)\right\Vert _{B(H)}^{2}\mathrm{d}t\bigg] =:\Lambda
_{\tau}\, ; \label{eq:estP}%
\end{equation}

\emph{iv)} for each $\tau\in\lbrack0,1)$ and any $\xi,\zeta\in L_{\mathcal{F}%
_{\tau}}^{4}(\Omega;H)$,\emph{ }
\[
\lim_{|s-t|\rightarrow0}\mathbb{E}\left\langle \xi,(P(s)-P(t))\zeta
\right\rangle _{H}=0\ \ \text{with}\ s,t\in\lbrack\tau,1].
\]

\end{theorem}

The proof of the above theorem depends on the following result in which the
OBSDE \eqref{eq:obsde} is considered in the Hilbert space $B_{2}(H)$. Such a
lemma is similar to Theorems 5.4 and 5.5 in Guatteri-Tessitore
\cite{guatteri2006backward}.

\begin{lemma}
\label{lem:HSmild} In addition to the conditions in Theorem \ref{thm:obsde},
suppose $P_{1}\in L_{\mathcal{F}_{1}}^{2}(\Omega;B_{2}(H))$ and $G\in
L_{\mathcal{P} }^{2}(\Omega\times\lbrack0,1];B_{2}(H))$. Then there exists a
unique mild solution $(P(\cdot),Q(\cdot))$ in the space
\[
\mathcal{C}_{\mathcal{P}}([0,1];L^{2}(\Omega;B_{2}(H)))\times L_{\mathcal{P}
}^{2}(\Omega\times\lbrack0,1];l^{2}(B_{2}(H))),
\]
that is, $(P,Q)$ satisfies relation \eqref{eq:mild} instead of $(P^{n},Q^{n}
)$. Moreover, the assertions (ii) and (iii) in Theorem \ref{thm:obsde} hold
true with respect to such $P(\cdot)$.
\end{lemma}

\begin{proof}
The existence and uniqueness of the (mild) solution $(P(\cdot),Q(\cdot))$ were
given by Theorem 5.4 in \cite{guatteri2006backward}. Next, we indicate that
Lemma \ref{lem:see} implies
\begin{equation}
\left\Vert y^{\tau,\xi}(\tau)\right\Vert _{H}^{4}\leq K(M_{A},M_{2}%
)\,\left\Vert \xi\right\Vert _{H}^{4}\ \ \ \ \text{(a.s.)} \label{eq:esty}%
\end{equation}
for any $\tau\in\lbrack0,1]$ and $\xi\in L_{\mathcal{F}_{\tau}}^{4}(\Omega
;H)$. In the case of $A\in B(H)$, one can show the relation \eqref{eq:exprP}
by the generalized It\^{o} formula (see e.g. \cite[Theorem 4.17]%
{da1992stochastic}). Then the general case can be obtained by a standard
argument of the Yoshida approximation. The inequality \eqref{eq:estP} follows
from \eqref{eq:exprP} and \eqref{eq:esty}.
\end{proof}

\medskip

\begin{proof}
[Proof of Theorem \ref{thm:obsde}]For the sake of convenience, we write the
right-hand side of equality \eqref{eq:exprP} as $T_{\tau}(\xi,\zeta;G,P_{1})$.
Then it follows from \eqref{eq:esty} and Young's inequality that for any
$\xi,\zeta\in L_{\mathcal{F}_{\tau}}^{4}(\Omega;H)$,
\[
\left\vert T_{\tau}(\xi,\zeta;G,P_{1})\right\vert
\leq\sqrt{\Lambda_{\tau}}\left\Vert \xi\right\Vert _{H}\left\Vert
\zeta\right\Vert _{H}\ \ \text{(a.s.)}.
\]
We shall always select an RCLL version of the martingale $(\Lambda_{\tau}%
;\tau\in\lbrack0,1])$. Fix arbitrary $\tau\in\lbrack0,1]$ and take a standard
complete orthonormal basis $\{e_{i}^{H}\}$ in $H$. Then there is a set of full
probability $\Omega_{1}\subset\Omega$ such that for each $\omega\in\Omega_{1}%
$,%
\[
\left\vert \left[  T_{\tau}(e_{i}^{H},e_{j}^{H};G,P_{1})\right]
(\omega)\right\vert \leq\sqrt{\Lambda_{\tau}(\omega)},\ \ \forall
i,j\in\mathbb{N}.
\]
Hence, from the Riesz representation theorem, there is a unique $P(\tau
,\omega)\in B(H)$ for each $\omega\in\Omega_{1}$ such that%
\begin{equation}
\langle e_{i}^{H},P(\tau,\omega)e_{j}^{H}\rangle_{H}=\left[
T_{\tau}(e_{i}^{H},e_{j}^{H};G,P_{1})\right]
(\omega),\ \ \forall i,j\in\mathbb{N}, \label{eq:controfP}%
\end{equation}
and
\[
\Vert P(\tau,\omega)\Vert_{B(H)}\leq\sqrt{\Lambda_{\tau}(\omega)}%
,\ \ \forall\omega\in\Omega_{1}.
\]
It is easy to check that $\langle x,P(\tau)y\rangle_{H}=[  T_{\tau}
(x,y;G,P_{1})]$ (a.s.) for any $x,y\in H$; and furthermore, for any simple
$H$-valued $\mathcal{F}_{\tau} $-measurable random variables $\xi,\zeta$, we
have
\[
\langle\xi,P(\tau)\zeta\rangle_{H}=\left[  T_{\tau}
(\xi,\zeta;G,P_{1})\right] ,\ \ (\textrm{a.s.})
\]
Then by a standard argument of approximation, the above relation holds true
for any $\xi,\zeta\in L_{\mathcal{F}_{\tau}}^{4}(\Omega;H)$.

We claim: the operator-valued process $P(\cdot)$ defined in
\eqref{eq:controfP} is the desired generalized solution. Indeed, for each
$n\in\mathbb{N}$, we introduce the finite dimensional projection $\Pi
_{n}:H\rightarrow H:x\rightarrow\sum_{i=1}^{n}\langle x,e_{i}^{H}\rangle
_{H}e_{i}^{H}$, define%
\[
G^{n}(t,\omega):=\Pi_{n}G(t,\omega)\Pi_{n}\ \text{ and \ }P_{1}^{n}%
(\omega):=\Pi_{n}P_{1}(\omega)\Pi_{n}%
\]
and find from Lemma \ref{lem:HSmild} the (unique) mild solution $(P^{n}%
,Q^{n})$ to OBSDE \eqref{eq:obsde} with the input $(G^{n},P_{1}^{n})$.
Obviously, the conditions (1) and (2) in Definition \ref{defn:ger} are
satisfied. Noticing the construction of $G^{n}$, it remains to show $\langle
x,P^{n}(t)y\rangle_{H}\rightarrow\langle x,P(t)y\rangle_{H}$ (a.s.) for any
$x,y\in H$ and $t\in\lbrack0,1]$. It follows from Lemma \ref{lem:HSmild} that
$\langle x,P^{n}(\tau)y\rangle_{H}=T_{\tau}(x,y;G^{n},P_{1}^{n})$ (a.s.).
By the Lebesgue dominated convergence theorem, we have%
\[
T_{\tau}(x,y;G^{n},P_{1}^{n})\rightarrow T_{\tau}(x,y;G,P_{1}%
)\ \ \ \ \text{(a.s.).}%
\]
This implies $\langle x,P^{n}(\tau)y\rangle_{H}\rightarrow\langle
x,P(\tau)y\rangle_{H}$. The claim is proved.

On the other hand, from Lemma \ref{lem:HSmild}, Definition \ref{defn:ger} and
the Lebesgue dominated convergence theorem, the relations \eqref{eq:exprP} and
\eqref{eq:estP} hold true for every generalized solution to OBSDE
\eqref{eq:obsde}, which yields assertions (ii) and (iii). Moreover, the
relation \eqref{eq:exprP} implies the uniqueness of the generalized solution.
Thus the assertion (i) is proved.

It remains to prove the assertion (iv). Fix arbitrary $\tau\in\lbrack0,1)$.
Without loss of generality, we assume $t<s$. Then for any $\xi,\zeta\in
L_{\mathcal{F}_{\tau}}^{4}(\Omega;H)$ we have (recall equation
\eqref{eq:y.eq})
\begin{align*}
\mathbb{E}\langle\xi &  ,(P(s)-P(t))\zeta\rangle_{H}\\
=~  &  \mathbb{E}\left[  \left\langle y^{s,\xi}(1),P_{1}y^{s,\zeta
}(1)\right\rangle _{H}-\left\langle y^{t,\xi}(1),P_{1}y^{t,\zeta
}(1)\right\rangle _{H}\right] \\
&  +\mathbb{E}\int_{s}^{1}\left[  \left\langle y^{s,\xi}(r),G(r)y^{s,\zeta
}(r)\right\rangle _{H}\mathrm{-}\,\left\langle y^{t,\xi}(r),G(r)y^{t,\zeta
}(r)\right\rangle _{H}\right]  \mathrm{d}r\\
&  -\mathbb{E}\int_{t}^{s}\left\langle y^{t,\xi}(r),G(r)y^{t,\zeta
}(r)\right\rangle _{H}\mathrm{d}r\\
=:\,  &  I_{1}+I_{2}+I_{3}%
\end{align*}
First, it follows from \eqref{eq:esty} and Young's inequality that
\begin{equation}
\left\vert I_{3}\right\vert ^{2}\leq K(M_{A},M_{2})\,\mathbb{E}\int_{t}%
^{s}\left\Vert G(r)\right\Vert _{\mathcal{B}(H)}^{2}\mathrm{d}r\cdot
\sqrt{\mathbb{E}\left\Vert \xi\right\Vert _{H}^{4}\cdot\mathbb{E}\left\Vert
\zeta\right\Vert _{H}^{4}}\rightarrow0, \label{eq:proof.301}%
\end{equation}
as$\ \left\vert s-t\right\vert \rightarrow0$. On the other hand, the
trajectory of $y^{\tau,\xi}(\cdot)$ is continuous in $H$, which along with
\eqref{eq:esty} and the Lebesgue dominated convergence theorem yields
\begin{align*}
\left\vert I_{2}\right\vert ^{2}~\leq~  &  K(M_{A},M_{2},G,P_{1}%
)\Big(\sqrt{\mathbb{E}\left\Vert \xi\right\Vert _{H}^{4}}\sqrt{\mathbb{E}%
\left\Vert \zeta-y^{t,\zeta}(s)\right\Vert _{H}^{4}}\\
&  +\sqrt{\mathbb{E}\left\Vert \zeta\right\Vert _{H}^{4}}\sqrt{\mathbb{E}%
\left\Vert \xi-y^{t,\xi}(s)\right\Vert _{H}^{4}}\,\Big)\rightarrow
0,\ \ \text{as}\ \left\vert s-t\right\vert \rightarrow0.
\end{align*}
Similarly, we can show
\[
|I_{1}|\rightarrow0,\ \ \ \ \text{as\ }\ \left\vert s-t\right\vert
\rightarrow0.
\]
Therefore, we have for any $\xi,\zeta\in L_{\mathcal{F}_{\tau}}^{4}(\Omega
;H)$,
\[
\lim_{|s-t|\rightarrow0}\mathbb{E}\left\langle \xi,(P(s)-P(t))\zeta
\right\rangle _{H} =0,\ \ \text{with}\ s,t\in\lbrack\tau,1].
\]
The assertion (iv) is proved. This concludes the theorem.
\end{proof}

\subsection{A basic property of the generalized solution}

The absence of $Q(\cdot)$ in the definition of generalized solution brings a
new difficulty in our derivation of the stochastic maximum principle compared
with the traditional duality approach (cf. \cite{peng1990general}). The
following result will play a key role to overcome the difficulty.

\begin{proposition}
\label{prop:P} Let the conditions in Theorem \ref{thm:obsde} be satisfied. For
any $\tau\in\lbrack0,1)$, $\vartheta,\theta\in L_{\mathcal{F}_{\tau}}%
^{4}(\Omega;l^{2}(H))$ and $\varepsilon\in(0,1-\tau)$, let $y_{\varepsilon
}^{\tau,\vartheta}(\cdot)$ be the mild solution to equation
\begin{align*}
y_{\varepsilon}^{\tau,\vartheta}(t)~=~  &  \int_{\tau}^{t}\mathrm{e}%
^{(t-s)A}A_{\sharp}(s)y_{\varepsilon}^{\tau,\vartheta}(s)\,\mathrm{d}s\\
&  +\int_{\tau}^{t}\mathrm{e}^{(t-s)A}\left[  C(s)y_{\varepsilon}%
^{\tau,\vartheta}(s)+\varepsilon^{-\frac{1}{2}}1_{[\tau,\tau+\varepsilon
)}\vartheta\right]  \cdot\mathrm{d}W_{s}.
\end{align*}
Then we have
\begin{align*}
&  \mathbb{E}\left\langle \vartheta,P(\tau)\theta\right\rangle _{l^{2}(H)}\\
&  =~\lim_{\varepsilon\downarrow0}\mathbb{E}\bigg[\left\langle y_{\varepsilon
}^{\tau,\vartheta}(1),P_{1}y_{\varepsilon}^{\tau,\theta}(1)\right\rangle
_{H}+\int_{\tau}^{1}\left\langle y_{\varepsilon}^{\tau,\vartheta
}(t),G(t)y_{\varepsilon}^{\tau,\theta}(t)\right\rangle _{H}\mathrm{d}%
t\bigg]\text{.}%
\end{align*}

\end{proposition}

\begin{remark}
In the $B_{2}(H)$-framework, the above result can be easily proved by the
It\^o formula (recalling Lemma \ref{lem:HSmild}); however, this approach
fails in the general $B(H)$-framework due to the absence of $Q(\cdot)$ in the
generalized solution. Besides, an approximation argument by using the
sequence $(P^{n},Q^{n} ,G^{n})$ seems also difficult to prove the above
result.
\end{remark}

The previous proposition follows from several lemmas. For the sake of
convenience, we denote%
\[
T_{\tau}^{\varepsilon}(\vartheta,\theta)=\mathbb{E}\left\langle y_{\varepsilon
}^{\tau,\vartheta}(1),P_{1}y_{\varepsilon}^{\tau,\theta}(1)\right\rangle
_{H}+\mathbb{E}\int_{\tau}^{1}\left\langle y_{\varepsilon}^{\tau,\vartheta
}(t),G(t)y_{\varepsilon}^{\tau,\theta}(t)\right\rangle _{H}\mathrm{d}t.
\]
Define
\begin{align*}
B_{A}(\tau) :=  &  \big\{\xi\in L_{\mathcal{F}_{\tau}}^{4}(\Omega
;H):\xi(\omega)\in D(A)\text{ }\\
&  ~~\text{and }\left\Vert \xi\right\Vert _{A} :=\sup_{\omega}\left(
\left\Vert A\xi\right\Vert _{H}+\left\Vert \xi\right\Vert _{H}\right)
<\infty\big\}
\end{align*}
which is dense in $L_{\mathcal{F}_{\tau}}^{4}(\Omega;H)$. Set $e_{i}%
=(0,\dots,0,1,0,\dots)$ with only the $i$-th element nonzero. Then $\xi
e_{i}\in L_{\mathcal{F}_{\tau}}^{4}(\Omega;l^{2}(H))$ for any $\xi\in
L_{\mathcal{F}_{\tau}}^{4}(\Omega;H)$ and $i\in\mathbb{N}$. Moreover, we
define%
\[
\xi^{i}(t):=\varepsilon^{-\frac{1}{2}}(W_{t}^{i}-W_{\tau}^{i})\xi
\ \ \text{and}\ \ \zeta^{j}(t):=\varepsilon^{-\frac{1}{2}}(W_{t}^{j}-W_{\tau
}^{j})\zeta.
\]

\begin{lemma}
\label{lem:appr} For any $\tau\in\lbrack0,1)$, $\xi\in B_{A}(\tau)$ and
$i\in\mathbb{N}$, we have%
\[
\mathbb{E}\Vert y_{\varepsilon}^{\tau,\xi e_{i}}(\tau+\varepsilon)-\xi
^{i}(\tau+\varepsilon)\Vert_{H}^{4}\leq K\varepsilon^{2}\Vert\xi\Vert_{A}%
^{4}.
\]

\end{lemma}

\begin{proof}
For simplicity, we set $y^{i}(t):=y_{\varepsilon}^{\tau,\xi e_{i}}(t)$. Then
for $t\in\lbrack\tau,\tau+\varepsilon],$%
\begin{align*}
(y^{i}-\xi^{i})(t)  &  =\int_{\tau}^{t}\mathrm{e}^{(t-s)A}\left[  A_{\sharp
}(s)(y^{i}-\xi^{i})(s)+(A+A_{\sharp}(s))\xi^{i}(s)\right]  \mathrm{d}s\\
&  +\int_{\tau}^{t}\mathrm{e}^{(t-s)A}\left[  C(s)(y^{i}-\xi^{i}%
)(s)+C(s)\xi^{i}(s)\right]  \cdot\mathrm{d}W_{t}.
\end{align*}
Then by Lemma \ref{lem:see}, we have%
\[
\mathbb{E}\left\Vert (y^{i}-\xi^{i})(\tau+\varepsilon)\right\Vert _{H}^{4}\leq
K\,\mathbb{E}\Big(\int_{\tau}^{\tau+\varepsilon}\left\Vert \xi^{i}%
(t)\right\Vert _{A}^{2}\,\mathrm{d}t\Big)^{2}\leq K\varepsilon^{2}\left\Vert
\xi\right\Vert _{A}^{4}.
\]
The lemma is proved.
\end{proof}

Notice the fact that for any $\xi,\zeta\in B_{A}(\tau)$,
\begin{align*}
T_{\tau}^{\varepsilon}(\xi e_{i},\zeta e_{j})~  &  =~\mathbb{E}\int_{\tau
}^{\tau+\varepsilon}\left\langle y_{\varepsilon}^{\tau,\xi e_{i}%
}(t),G(t)y_{\varepsilon}^{\tau,\zeta e_{j}}(t)\right\rangle _{H}%
\,\mathrm{d}t\\
&  ~~~~+\mathbb{E}\left\langle y_{\varepsilon}^{\tau,\xi e_{i}}(\tau
+\varepsilon),P(\tau+\varepsilon)y_{\varepsilon}^{\tau,\zeta e_{j}}%
(\tau+\varepsilon)\right\rangle _{H}\\
&  =:~J_{1}+J_{2}.
\end{align*}
Now we let $\varepsilon$ tend to $0$. On the one hand, one can show that the
term $J_{1}$ tends to $0$ similarly as in \eqref{eq:proof.301}; on the other
hand, by means of Lemma \ref{lem:appr}, the term $J_{2}$ should tend to the
same limit as $\mathbb{E}\left\langle \xi^{i}(\tau+\varepsilon),P(\tau
+\varepsilon)\zeta^{j}(\tau+\varepsilon)\right\rangle _{H}$. Indeed, we have

\begin{lemma}
\label{lem:PandT} For any $\tau\in\lbrack0,1)$, $\xi,\zeta\in B_{A}(\tau)$ and
$i,j\in\mathbb{N}$, we have%
\[
\lim_{\varepsilon\downarrow0}\left\{  \mathbb{E}\left\langle \xi^{i}%
(\tau+\varepsilon),P(\tau+\varepsilon)\zeta^{j}(\tau+\varepsilon)\right\rangle
_{H}-T_{\tau}^{\varepsilon}(\xi e_{i},\zeta e_{j})\right\}  =0.
\]

\end{lemma}

\begin{proof}
It is sufficient to show
\[
\lim_{\varepsilon\downarrow0}\left\{  \mathbb{E}\left\langle \xi^{i}%
(\tau+\varepsilon),P(\tau+\varepsilon)\zeta^{j}(\tau+\varepsilon)\right\rangle
_{H}-J_{2}\right\}  =0.
\]
Indeed, from Theorem \ref{thm:obsde}, Lemmas \ref{lem:see} and \ref{lem:appr},
we have
\begin{align*}
&  \left\vert \mathbb{E}\left\langle \xi^{i}(\tau+\varepsilon),P(\tau
+\varepsilon)\zeta^{j}(\tau+\varepsilon)\right\rangle _{H}-J_{2}\right\vert \\
&  \leq K\left(  \mathbb{E}\Vert\xi\Vert_{H}^{4}\right)  ^{\frac{1}{4}}\left(
\mathbb{E}\Vert y_{\varepsilon}^{\tau,\zeta e_{j}}(\tau+\varepsilon)-\zeta
^{j}(\tau+\varepsilon)\Vert_{H}^{4}\right)  ^{\frac{1}{4}}\\
&  ~~~+K\left(  \mathbb{E}\Vert\zeta\Vert_{H}^{4}\right)  ^{\frac{1}{4}%
}\left(  \mathbb{E}\Vert y_{\varepsilon}^{\tau,\xi e_{i}}(\tau+\varepsilon
)-\xi^{i}(\tau+\varepsilon)\Vert_{H}^{4}\right)  ^{\frac{1}{4}}\\
&  \rightarrow0,\ \ \ \ \text{as\ }\ \varepsilon\downarrow0.
\end{align*}
This concludes this lemma.
\end{proof}

On the other hand, from the continuity of $P(\cdot)$ we have the following

\begin{lemma}
\label{lem:conP} For any $\tau\in\lbrack0,1)$, $\xi,\zeta\in B_{A}(\tau)$ and
$i,j\in\mathbb{N}$, we have%
\[
\lim_{\varepsilon\downarrow0}\mathbb{E}\left\langle \xi^{i}(\tau
+\varepsilon),P(\tau+\varepsilon)\zeta^{j}(\tau+\varepsilon)\right\rangle
_{H}=\mathbb{E}\left\langle \xi,P(\tau)\zeta\right\rangle _{H}.
\]

\end{lemma}

\begin{proof}
It is easy to see
\[
\mathbb{E}\left\langle \xi^{i}(\tau+\varepsilon),P(\tau)\zeta^{j}%
(\tau+\varepsilon)\right\rangle _{H}=\mathbb{E}\left\langle \xi,P(\tau
)\zeta\right\rangle _{H}.
\]
Thus we need prove%
\[
\lim_{\varepsilon\downarrow0}\mathbb{E}\left\langle \xi^{i}(\tau
+\varepsilon),[P(\tau+\varepsilon)-P(\tau)]\zeta^{j}(\tau+\varepsilon
)\right\rangle _{H}=0.
\]

It follows from \eqref{eq:estP}, the boundedness of $\xi,\zeta$, and Doob's
martingale inequality (cf. \cite{doob1953stochastic}) that (recall
\eqref{eq:estP})
\[
\left\vert \left\langle \xi,\left[  P(\tau+\varepsilon)-P(\tau)\right]
\zeta\right\rangle _{H}\right\vert ^{2}\leq 4 \big(\max\nolimits_{t\in
[0,1]}\Lambda_{t}\big) \left\Vert \xi\right\Vert _{H} ^{2}\left\Vert
\zeta\right\Vert _{H}^{2}\in L^{1}(\Omega),
\]
Note that
\begin{align*}
&  \left\langle \xi^{i}(\tau+\varepsilon),[P(\tau+\varepsilon)-P(\tau
)]\zeta^{j}(\tau+\varepsilon)\right\rangle _{H}\\
&  =\varepsilon^{-1}(W_{\tau+\varepsilon}^{i}-W_{\tau}^{i})(W_{\tau
+\varepsilon}^{j}-W_{\tau}^{j})\left\langle \xi,\left[  P(\tau+\varepsilon
)-P(\tau)\right]  \zeta\right\rangle _{H}%
\end{align*}
Then from Theorem \ref{thm:obsde}(iv) and the Lebesgue dominated convergence
theorem, we have
\begin{align*}
&  \left\vert \mathbb{E}\left\langle \xi^{i}(\tau+\varepsilon),[P(\tau
+\varepsilon)-P(\tau)]\zeta^{j}(\tau+\varepsilon)\right\rangle _{H}\right\vert
^{2}\\
&  \leq\varepsilon^{-2}\cdot\mathbb{E}\big(|W_{\tau+\varepsilon}^{i}-W_{\tau
}^{i}|^{2}|W_{\tau+\varepsilon}^{j}-W_{\tau}^{j}|^{2}\big)\cdot\mathbb{E}%
\left\vert \left\langle \xi,\left[  P(\tau+\varepsilon)-P(\tau)\right]
\zeta\right\rangle _{H}\right\vert ^{2}\\
&  \leq3\mathbb{E}\left\vert \left\langle \xi,\left[  P(\tau+\varepsilon
)-P(\tau)\right]  \zeta\right\rangle _{H}\right\vert ^{2}\rightarrow
0,\ \ \ \ \text{as}\ \ \varepsilon\downarrow0.
\end{align*}
The lemma is proved.
\end{proof}

Now we are in a position to complete the proof of Proposition \ref{prop:P}.

\begin{proof}
[Proof of Proposition \ref{prop:P}]Fix any $\vartheta,\theta\in L_{\mathcal{F}%
_{\tau}}^{4}(\Omega;l^{2}(H))$. For arbitrary $\delta>0$, we can find (from
the density) a large $m_{\delta}$ and $\{\xi_{i},\zeta_{i}\,;\,i=1,\dots
,m_{\delta}\}$ $\subset B_{A}(\tau)$ such that%
\begin{align*}
&  \vartheta_{m_{\delta}}:=(\xi_{1},\dots,\xi_{m_{\delta}}),\ \ \theta
_{m_{\delta}}:=(\zeta_{1},\dots,\zeta_{m_{\delta}}),\\
&  \mathbb{E}\left\Vert \vartheta-\vartheta_{m_{\delta}}\right\Vert
_{l^{2}(H)}^{4}+\mathbb{E}\left\Vert \theta-\vartheta_{m_{\delta}}\right\Vert
_{l^{2}(H)}^{4}<\delta^{4}.
\end{align*}
Then we have%
\begin{align*}
\big\vert\mathbb{E}\left\langle \vartheta,P(\tau)\theta\right\rangle
_{l^{2}(H)}-\mathbb{E}\left\langle \vartheta_{m_{\delta}},P(\tau
)\theta_{m_{\delta}}\right\rangle _{l^{2}(H)}\big\vert  &  <K(\vartheta
,\theta,P)\,\delta,\\
\left\vert T_{\tau}^{\varepsilon}(\vartheta,\theta)-T_{\tau}^{\varepsilon
}(\vartheta_{m_{\delta}},\theta_{m_{\delta}})\right\vert  &  <K(\vartheta
,\theta,P)\,\delta.
\end{align*}
On the other hand, from Lemmas \ref{lem:PandT} and \ref{lem:conP}, one can
easily check that
\[
\mathbb{E}\left\langle \vartheta_{m_{\delta}},P(\tau)\theta_{m_{\delta}%
}\right\rangle _{l^{2}(H)}=\lim_{\varepsilon\downarrow0}T_{\tau}^{\varepsilon
}(\vartheta_{m_{\delta}},\theta_{m_{\delta}}).
\]
Thus we have%
\[
\limsup_{\varepsilon\downarrow0}\big\vert\mathbb{E}\left\langle \vartheta
,P(\tau)\theta\right\rangle _{l^{2}(H)}-T_{\tau}^{\varepsilon}(\vartheta
,\theta)\big\vert < K(\vartheta,\theta,P)\,\delta.
\]
From the arbitrariness of $\delta$, we conclude the proposition.
\end{proof}

\section{The stochastic maximum principle and its proof}

\subsection{The statement of the main theorem}

Now we are in a position to formulate the stochastic maximum principle for
optimal controls. Define the \emph{Hamiltonian}
\[
\mathcal{H}:\,\Omega\times[0,1]\times H\times U\times H\times l^{2}(H)
\to\mathbb{R},
\]
as the form
\begin{equation}
\mathcal{H}(t,x,u,p,q):=l(t,x,u)+\left\langle p,f(t,x,u)\right\rangle
_{H}+\left\langle q,g(t,x,u)\right\rangle _{l^{2}(H)}, \label{eq:hamiltonian}%
\end{equation}
then our main result can be stated as follows.

\begin{theorem}
[stochastic maximum principle]\label{thm:smp} Let Assumptions \ref{ass:A} and
\ref{ass:fglh} be satisfied, $\bar{x}(\cdot)$ be the state process with
respect to an optimal control $\bar{u}(\cdot)$. Then for each $u\in U$, the
\emph{variational inequality}
\begin{align*}
0\,\leq\,  &  \mathcal{H}(\tau,\bar{x}(\tau),u,p(\tau),q(\tau))-\mathcal{H}%
(\tau,\bar{x}(\tau),\bar{u}(\tau),p(\tau),q(\tau))\\
&  +\frac{1}{2}\left\langle g(\tau,\bar{x}(\tau),u)-g(\tau,\bar{x}(\tau
),\bar{u}(\tau)),P(\tau)[g(\tau,\bar{x}(t),u)-g(\tau,\bar{x}(\tau),\bar
{u}(\tau))]\right\rangle _{l^{2}(H)}%
\end{align*}
holds for a.e. $(\tau,\omega)\in\lbrack0,1)\times\Omega$, where $(p(\cdot
),q(\cdot))$ is the solution to BSEE \eqref{eq:bsee} with
\[
F(t,p,q)=\mathcal{H}_{x}(t,\bar{x}(t),\bar{u}(t),p,q)\,,\qquad\xi=h_{x}%
(\bar{x}(1)),
\]
and $P(\cdot)$ is the generalized solution to OBSDE \eqref{eq:obsde} with%
\begin{align*}
&  A_{\sharp}(t)=f_{x}(t,\bar{x}(t),\bar{u}(t)),\qquad C(t)=g_{x}(t,\bar
{x}(t),\bar{u}(t)),\\
&  G(t)=\mathcal{H}_{xx}(t,\bar{x}(t),\bar{u}(t),p(t),q(t)),\qquad
P_{1}=h_{xx}(\bar{x}(1)).
\end{align*}

\end{theorem}

\subsection{Proof of the main theorem}

The proof is divided into the following three steps.

\emph{Step 1. The spike variation and second-order expansion.}

The approach in this step is quite standard (cf. \cite{peng1990general}%
). Recall that $\bar{x}(\cdot)$ is the state process with respect to a
optimal control $\bar{u}(\cdot)$. We construct a perturbed admissible control
in the
following way%
\[
u^{\varepsilon}(t):=\left\{
\begin{array}
[c]{ll}%
u, & \text{if }t\in\lbrack\tau,\tau+\varepsilon],\\
\bar{u}(t), & \text{otherwise,}%
\end{array}
\right.
\]
with fixed $\tau\in\lbrack0,1)$, sufficiently small positive $\varepsilon$,
and an arbitrary $U$-valued $\mathcal{F}_{\tau}$-measurable random variable
$u$ satisfying $\mathbb{E}|u|_{U}^{4}<\infty$.

Let $x^{\varepsilon}(\cdot)$ be the state process with respect to control
$u^{\varepsilon}(\cdot)$. For the sake of convenience, we denote for
$\varphi=f,g,l,f_{x},g_{x},l_{x},f_{xx},g_{xx},l_{xx},$
\begin{align*}
\bar{\varphi}(t) &  :=\varphi(t,\bar{x}(t),\bar{u}(t)),\\
\varphi^{\Delta}(t) &  :=\varphi(t,\bar{x}(t),u^{\varepsilon}(t))-\bar
{\varphi}(t,\bar{x}(t),\bar{u}(t)),
\end{align*}
Let $x_{1}(\cdot)$ and $x_{2}(\cdot)$ be the solutions respectively to%
\begin{align*}
x_{1}(t)= &  \int_{0}^{t}\mathrm{e}^{(t-s)A}\bar{f}_{x}(s)x_{1}(s)\,\mathrm{d}%
s+\int_{0}^{t}\mathrm{e}^{(t-s)A}[\bar{g}_{x}(s)x_{1}(s)+g^{\Delta}%
(s)]\cdot\mathrm{d}W_{s},\\
x_{2}(t)= &  \int_{0}^{t}\mathrm{e}^{(t-s)A}\Big[\bar{f}_{x}(s)x_{2}%
(s)+\frac{1}{2}\bar{f}_{xx}(s)\left(  x_{1}\otimes x_{1}\right)
(s)+f^{\Delta}(s)\Big]\,\mathrm{d}s\\
&  +\int_{0}^{t}\mathrm{e}^{(t-s)A}\Big[\bar{g}_{x}(s)x_{2}(s)+\frac{1}{2}%
\bar{g}_{xx}(s)\left(  x_{1}\otimes x_{1}\right)  (s)+g_{x}^{\Delta}%
(s)x_{1}(s)\Big]\,\cdot\mathrm{d}W_{s}.
\end{align*}
It follows from Lemma \ref{lem:see} that
for all $t\in\lbrack0,1],$%
\begin{eqnarray}
\left\{~\begin{split} &\varepsilon^{-2}\mathbb{E}\left\Vert
x_{1}(t)\right\Vert _{H}^{4} +\varepsilon^{-1}\mathbb{E}\left\Vert
x_{1}(t)\right\Vert _{H}^{2}
+\varepsilon^{-2}\mathbb{E}\left\Vert x_{2}(t)\right\Vert _{H}^{2}\leq K,\\
&\varepsilon^{-2}\mathbb{E}\left\Vert
x^{\varepsilon}(t)-\bar{x}(t)\right\Vert
_{H}^{4}+\varepsilon^{-1}\mathbb{E}\left\Vert x^{\varepsilon}(t)-\bar
{x}(t)\right\Vert _{H}^{2}\leq K,\\
&\varepsilon^{-2}\mathbb{E}\left\Vert x^{\varepsilon}(t)-\bar{x}(t)-x_{1}
(t)\right\Vert _{H}^{2}\leq K,\\
&\varepsilon^{-2}\mathbb{E}\left\Vert
x^{\varepsilon}(t)-\bar{x}(t)-x_{1}(t)-x_{2} (t)\right\Vert _{H}^{2}=o(1).
\end{split}
\right.  \label{eq:err}%
\end{eqnarray}
This along with the fact
\[
J(u^{\varepsilon}(\cdot))-J(\bar{u}(\cdot))\geq0
\]
yields (for details, we refer to \cite{yong1999stochastic} or
\cite{du2012general})
\begin{align*}
o(\varepsilon)~ \leq&  ~\mathbb{E}\int_{0}^{1}\Big[l^{\Delta}(t)+\left\langle
\bar{l}_{x}(t),x_{1}(t)+x_{2}(t)\right\rangle _{H}+\frac{1}{2}\left\langle
x_{1}(t),\bar{l}_{xx}(t)x_{1}(t)\right\rangle _{H}\Big]\mathrm{d}t\\
&  +\mathbb{E}\left\langle h_{x}(\bar{x}(1)),x_{1}(1)+x_{2}(1)\right\rangle
_{H}+\frac{1}{2}\left\langle x_{1}(1),h_{xx}(\bar{x}(1))x_{1}(1)\right\rangle
_{H}.
\end{align*}

\emph{Step 2. First-order duality analysis. }

It follows from Lemma \ref{lem:see}(2) that BSEE \eqref{eq:bsee} has a unique
solution $(p(\cdot),q(\cdot))$ in this situation. Recalling the Hamiltonian
\eqref{eq:hamiltonian}, and from Lemma \ref{lem:duality}, we have
\begin{align*}
&  \mathbb{E}\int_{0}^{1}\left[l^{\Delta}(t)+\left\langle \bar{l}%
_{x}(t),x_{1}(t)+x_{2}(t)\right\rangle_{H}\right]  \mathrm{d}%
t+\mathbb{E}\left\langle h_{x}(\bar{x}(1)),x_{1}(1)+x_{2}(1)\right\rangle
_{H}\\
=~  &  o(\varepsilon)+\mathbb{E}\int_{\tau}^{\tau+\varepsilon}\left[
\mathcal{H}(t,\bar{x}(t),u,p(t),q(t))-\mathcal{H}(t,\bar{x}(t),\bar
{u}(t),p(t),q(t))\right]  \,\mathrm{d}t\\
&  +\frac{1}{2}\mathbb{E}\int_{0}^{1}\left[  \left\langle p(t),\bar{f}%
_{xx}(t)\left(  x_{1}\otimes x_{1}\right)  (t)\right\rangle _{H}+\left\langle
q(t),\bar{g}_{xx}(t)\left(  x_{1}\otimes x_{1}\right)  (t)\right\rangle
_{l^{2}(H)}\right]  \,\mathrm{d}t.
\end{align*}
Hence, we get%
\begin{align}
o(1)~\leq~  &  \varepsilon^{-1}\mathbb{E}\int_{\tau}^{\tau+\varepsilon}\left[
\mathcal{H}(t,\bar{x}(t),u,p(t),q(t))-\mathcal{H}(t,\bar{x}(t),\bar
{u}(t),p(t),q(t))\right]  \,\mathrm{d}t\nonumber\label{eq:varitionineq}\\
&  +\frac{1}{2}\varepsilon^{-1}\mathbb{E}\int_{0}^{1}\left\langle
x_{1}(t),\mathcal{H}_{xx}(t,\bar{x}(t),\bar{u}(t),p(t),q(t))x_{1}%
(t)\right\rangle _{H}\,\mathrm{d}t\nonumber\\
&  +\frac{1}{2}\varepsilon^{-1}\mathbb{E}\left\langle x_{1}(1),h_{xx}(\bar
{x}(1))x_{1}(1)\right\rangle _{H}.
\end{align}

\emph{Step 3. Second-order duality analysis and completion of the proof.}

This is the key step in the proof. From Theorem \ref{thm:obsde}, it is easy
to check that OBSDE \eqref{eq:obsde} has a unique generalized solution
$P(\cdot)$ in this situation. Now we introduce the following equation
\begin{align*}
y^{\varepsilon}(t) =  &  \int_{\tau}^{t}\mathrm{e}^{(t-s)A}A_{\sharp
}(s)y^{\varepsilon}(s)\mathrm{d}s\\
&  +\int_{\tau}^{t}\mathrm{e}^{(t-s)A}[C(s)y^{\varepsilon}(s)+\varepsilon
^{-\frac{1}{2}}1_{[\tau,\tau+\varepsilon
)}g^{\Delta}(\tau)]\cdot\,\mathrm{d}W_{s},\ \ t \in[\tau,1].
\end{align*}
Then we have

\begin{lemma}\label{lem:leb}
For a.e. $\tau\in\lbrack0,1)$, we have%
\[
\lim_{\varepsilon\downarrow0}\sup_{t\in\lbrack\tau,1]}\mathbb{E}\big\Vert
\varepsilon ^{-\frac{1}{2}}x_{1}(t)-y^{\varepsilon}(t)\big\Vert_{H}^{4}=0.\
\]
\end{lemma}

\begin{proof}
By Lemma \ref{lem:see}, we have for each $\tau\in\lbrack0,1]$,%
\[
\sup_{t\in\lbrack\tau,1]}\mathbb{E}\big\Vert \varepsilon
^{-\frac{1}{2}}x_{1}(t)-y^{\varepsilon
}(t)\big\Vert_{H}^{4}\leq K\cdot\frac{1}{\varepsilon}\int_{\tau}%
^{\tau+\varepsilon}\mathbb{E}\left\Vert g^{\Delta}(t)-g^{\Delta}%
(\tau)\right\Vert _{l^{2}(H)}^{4}\mathrm{d}t.
\]
From the Lebesgue differentiation theorem, we have for each $X\in
L_{\mathcal{F}_{1}}^{4}(\Omega;l^{2}(H))$,
\[
\lim_{\varepsilon\downarrow0}\frac{1}{\varepsilon}\int_{\tau}^{\tau
+\varepsilon}\mathbb{E}\left\Vert g^{\Delta}(t)-X\right\Vert _{l^{2}(H)}%
^{4}\mathrm{d}t=\mathbb{E}\left\Vert g^{\Delta}(\tau)-X\right\Vert _{l^{2}%
(H)}^{4},\newline\ \ \text{for a.e. }\tau\in\lbrack0,1).
\]
Since $L_{\mathcal{F}_{1}}^{4}(\Omega;l^{2}(H))$ is separable, let $X$ run
through a countable density subset $Q$ in $L_{\mathcal{F}_{1}}^{4}%
(\Omega;l^{2}(H))$, and denote
\[
E:=\bigcup E_{X}:=\bigcup\big\{\tau:\text{the above relation does not hold for
}X\big\}.
\]
Then we have $\mathrm{meas}(E)=0$. For arbitrary positive $\eta$, take an
$X\in Q$ such that
\[
\mathbb{E}\left\Vert g^{\Delta}(\tau)-X\right\Vert _{l^{2}(H)}^{4}<\eta,
\]
then for each $\tau\in\lbrack0,1)\backslash E$,
\begin{align*}
&  \lim_{\varepsilon\downarrow0}\frac{1}{\varepsilon}\int_{\tau}%
^{\tau+\varepsilon}\mathbb{E}\left\Vert g^{\Delta}(t)-g^{\Delta}%
(\tau)\right\Vert _{l^{2}(H)}^{4}\,\mathrm{d}t\\
&  \leq\lim_{\varepsilon\downarrow0}\frac{8}{\varepsilon}\int_{\tau}%
^{\tau+\varepsilon}\mathbb{E}\left\Vert g^{\Delta}(t)-X\right\Vert _{l^{2}%
(H)}^{4}\mathrm{d}t+8\mathbb{E}\left\Vert g^{\Delta}(\tau)-X\right\Vert _{l^{2}(H)}%
^{4}\\
&  \leq16\mathbb{E}\left\Vert g^{\Delta}(\tau)-X\right\Vert _{l^{2}(H)}%
^{4}<16\eta.
\end{align*}
From the arbitrariness of $\eta$, we conclude this lemma.
\end{proof}

From the the above lemma, we have%
\begin{align*}
&  \varepsilon^{-1}\mathbb{E}\int_{0}^{1}\left\langle x_{1}(t),G(t)x_{1}%
(t)\right\rangle \,_{H}\mathrm{d}t+\varepsilon^{-1}\mathbb{E}\left\langle
x_{1}(1),P_{1}x_{1}(1)\right\rangle _{H}\\
&  =o(1)+\mathbb{E}\int_{\tau}^{1}\left\langle y^{\varepsilon}%
(t),G(t)y^{\varepsilon}(t)\right\rangle _{H}\,\mathrm{d}t+\mathbb{E}%
\left\langle y^{\varepsilon}(1),P_{1}y^{\varepsilon}(1)\right\rangle
_{H},\ \ \forall\tau\in\lbrack0,1)\backslash E.
\end{align*}
Keeping in mind the above relation, and applying Proposition \ref{prop:P}, we
conclude for each $\tau\in\lbrack0,1)\backslash E$,
\begin{align*}
&  \mathbb{E}\left\langle g^{\Delta}(\tau),P(\tau)g^{\Delta}(\tau
)\right\rangle _{l^{2}(H)}\\
&  =~\lim_{\varepsilon\downarrow0}\varepsilon^{-1}\left\{  \mathbb{E}\int
_{0}^{1}\left\langle x_{1}(t),G(t)x_{1}(t)\right\rangle _{H}\,\mathrm{d}%
t+\mathbb{E}\left\langle x_{1}(1),P_{1}x_{1}(1)\right\rangle _{H}\right\}  .
\end{align*}
This along with \eqref{eq:varitionineq} and the Lebesgue differentiation
theorem yields for each $u\in U$,
\begin{align*}
0~\leq~  &  \mathbb{E}\left[  \mathcal{H}(\tau,\bar{x}(\tau),u,p(\tau
),q(\tau))-\mathcal{H}(\tau,\bar{x}(\tau),\bar{u}(\tau),p(\tau),q(\tau
))\right] \\
&  +\frac{1}{2}\mathbb{E}\left\langle g^{\Delta}(\tau),P(\tau)g^{\Delta}%
(\tau)\right\rangle _{l^{2}(H)},\ \ \ \ \ \ \text{a.e. }\tau\in\lbrack0,1).
\end{align*}
Therefore, the desired variational inequality follows from a standard
argument (cf. \cite{kushner1972necessary}). This completes the proof of the
stochastic maximum principle.

\end{document}